\documentclass[12pt]{article}
\usepackage{geometry} % see geometry.pdf on how to lay out the page. There's lots.
\usepackage{amsfonts}
\usepackage{amsmath}
\usepackage{tipa}
\usepackage{fullpage}

\newtheorem{theorem}{Theorem}

\newtheorem{corollary}{Corollary}

\newtheorem{definition}{Definition}
\newtheorem{example}{Example}

\newtheorem{lemma}{Lemma}

\newtheorem{proposition}[theorem]{Proposition}

\newenvironment{proof}[1][Proof]{\begin{trivlist}
\item[\hskip \labelsep {\bfseries #1}]}{\end{trivlist}}

\begin{document}

\title{Further Pieri-type formulas for the nonsymmetric Macdonald polynomial}
\author{W. Baratta \\
	Department of Mathematics, University of Melbourne}
	\maketitle

\begin{abstract}
The branching coefficients in the expansion of the elementary symmetric function multiplied by a symmetric Macdonald polynomial $P_\kappa(z)$ are known explicitly. These formulas generalise the known $r=1$ case of the Pieri-type formulas for the nonsymmetric Macdonald polynomials $E_\eta(z)$. In this paper we extend beyond the case $r=1$ for the nonsymmetric Macdonald polynomials, giving the full generalisation of the Pieri-type formulas for symmetric Macdonald polynomials. The decomposition also allows the evaluation of the generalised binomial coefficients $\tbinom{\eta }{\nu }_{q,t}$ associated with the nonsymmetric Macdonald polynomials.
\end{abstract}

\section{Introduction\label{introduction}} 
In 1988 Ian Macdonald \cite{macdonaldagain} introduced the symmetric Macdonald polynomials $P_\kappa(z;q,t)$, a polynomial generalising, for example, the Schur and symmetric Jack polynomials. Six years later a generalisation of the symmetric Macdonald polynomials, the nonsymmetric Macdonald polynomials, were introduced \cite{cherednik, affine}. 

The symmetric Macdonald polynomials can be generated from the nonsymmetric Macdonald polynomials by a process of symmetrisation. Consequently properties of the nonsymmetric Macdonald polynomials imply corresponding properties of the symmetric polynomials, and can be used to both illuminate and simplify the theory of the latter \cite{marshallmacdonald}.  

The converse however does not always hold. For example, the nonsymmetric analogue of the
Pieri-type formula \cite[Section VI. 6]{macdonald}
\begin{equation}
e_{r}( z) P_{\kappa }( z;q,t) =\sum_{\lambda
}\psi _{\lambda /\kappa }P_{\lambda }( z;q,t)  \label{pieri}
\end{equation}
giving the explicit form of the branching coefficients $\psi _{\lambda
/\kappa }$, for the product of $P_{\kappa }( z;q,t) $ with the $
r^{th}$ elementary symmetric function, 
\begin{equation*}
e_{r}( z) =\sum_{1\leq i_{1}<\ldots<i_{r}\leq
n}z_{i_{1}}\ldots z_{i_{r}} \label{rth elementary},
\end{equation*}
is only known in the cases $r=1$ and $r=n-1$ \cite{wendyb}. In $( \ref{pieri}) $ the sum
is over $\lambda $ such that $\lambda /\kappa $ is a vertical $m$-strip and $\psi _{\lambda
/\kappa }$ is given by
\begin{equation*} \label{therequiredeval}
\psi _{\lambda /\kappa }:=t^{n(\lambda)-n(\kappa)}\frac{P_\kappa(t^{\underline{\delta }})}{P_\lambda(t^{\underline{\delta }})}\prod_{1\leq i < j \leq n}\frac{1-q^{\kappa_{i}-\kappa_{j}}t^{j-i+\theta_i-\theta_j}}{1-q^{\kappa_i-\kappa_j}t^{j-i}},
\end{equation*}
where $\theta\lambda-\kappa$ and $n(\lambda)=\sum_i(i-1)\lambda_i$. We remark that the Pieri formulas have been of recent interest in studies of
certain vanishing properties of Macdonald polynomials at $t^{k+1}q^{r-1}=1$ \cite{japanese}. The dual of $\left( \ref{pieri}\right) $ has also
found application in the study of certain probabilistic models related to
the Robinson-Schensted-Knuth corresponence \cite{lastpeter}.

In this work we provide explicit formulas for the branching
coefficients in the general Pieri-type formulas
\begin{equation} \label{nonzero}
e_{r}\left( z\right) E_{\eta }\left( z;q^{-1 },t^{-1 }\right)=\sum_{\lambda:|\lambda|=|\eta|+r}A_{\eta\lambda}^{(r)}E_\lambda(z;q^{-1},t^{-1}).
\end{equation}
As in \cite{wendyb} the coefficients are obtained via exploitation of the theory of interpolation polynomials. The theory of the latter is revised in Section \ref{insm}, after an account of the required theory of nonsymmetric Macdonald polynomials.
Also given in Section \ref{insm} is an alternative derivation of the known \cite{knop} extra vanishing properties of the interpolation Macdonald polynomials that play a key role in the study of the general Pieri-type formulas. 

We commence our study of the general Pieri-type coefficients in Section \ref{structure} by determining the necessary conditions for the coefficients $A_{\eta \lambda}^{(r)}$ in (\ref{nonzero}) to be non-zero. Explicit formulas for the branching coefficients are derived in Sections \ref{pieri1} and \ref{pierir}. We begin with an alternative derivation of the $r=1$ case in Sections \ref{pieri1} and proceed with the general formulas in Section \ref{pierir}. The main results are stated in Theorem \ref{generalpiericoefficients} and Proposition \ref{binomialcoeff}. As a corollary of Proposition \ref{binomialcoeff} we are able provide an explicit formula for the generalised binomial coefficient $\tbinom{\eta }{\nu }$. We conclude the paper by considering possible simplifications of the coefficient formulas and a discussion of yet another representation of the coefficients in the case of $r=1$ that could lead to more succinct expressions for the coefficients in general $r$ case.

\section{The nonsymmetric Macdonald polynomials}\label{nsm}

The nonsymmetric Macdonald polynomials $E_{\eta }:=E_{\eta }\left(
z;q,t\right) $, labelled by compositions $\eta :=\left( \eta _{1},...,\eta _{n}\right) $ are most commonly defined to be the simultaneous eigenfunctions of commuting Cherednik type operators \cite{affine}. Here we take an alternative approach by introducing two elementary operators that allow the polynomials to be generated, and consequently defined, recursively.

We first introduce the Demazure-Lustig operator, $T_i$, a switching type operator that relates the polynomials $E_{\eta}$ and $E_{s_i\eta}$. The operator $s_i$ is a transposition operator which acts on compositions 
\begin{equation*}
 s_{i} \eta
:=\left( \eta_{1},..., \eta_{i+1}, \eta_{i},..., \eta_{n}\right),  \label{switchingcomp}
\end{equation*}
and functions 
\begin{equation*}
\left( s_{i}f\right) \left( z_{1},...,z_{i},z_{i+1},...,z_{n}\right)
:=f\left( z_{1},...,z_{i+1},z_{i},...,z_{n}\right) .  \label{switching}
\end{equation*}
The operator $T_i$ is defined by
\begin{equation}
T_{i}:=t+\frac{tz_{i}-z_{i+1}}{z_{i}-z_{i+1}}\left( s_{i}-1\right) ,
\label{Ti}
\end{equation}
and is realisation of the type-A Hecke algebra 
\begin{align}
\left( T_{i}+1\right) \left( T_{i}-t\right) =&\hspace{0.1cm}0  \notag \\
T_{i}T_{i+1}T_{i} =&\hspace{0.1cm}T_{i+1}T_{i}T_{i+1},\hspace{1cm}i=2,...,n-2
\label{hecke algebra} \\
T_{i}T_{j} =&\hspace{0.1cm}T_{j}T_{i},\hspace{2.1cm}\left\vert i-j\right\vert >1.\text{ } 
\notag
\end{align}
The operator $T_i$ acts on $E_\eta$ according to \cite{mimachi}
\begin{equation}
T_{i}E_{\eta }(z)=\left\{ 
\begin{tabular}{ll}
$ \frac{t-1}{1-\delta _{i,\eta }^{-1}( q,t) }
E_{\eta }(z)+tE_{s_{i}\eta }(z)$ & $\eta _{i}<\eta _{i+1}$ \\ 
$tE_{\eta }(z)$ & $\eta _{i}=\eta _{i+1}$ \\ 
$ \frac{t-1}{1-\delta _{i,\eta }^{-1}( q,t) }
E_{\eta }(z)+\frac{( 1-t\delta _{i,\eta }( q,t) ) (
1-t^{-1}\delta _{i,\eta }( q,t) ) }{( 1-\delta
_{i,\eta }( q,t) ) ^{2}}E_{s_{i}\eta }(z)$ & $\eta _{i}>\eta
_{i+1}$.
\end{tabular}
\right.  \label{TiEn}
\end{equation}
In (\ref{TiEn}) $\delta _{i,\eta }( q,t) :=\overline{\eta }_i/\overline{
\eta} _{i+1}$, with
\begin{equation}
\overline{\eta }_{i}:=q^{\eta _{i}}t^{-l_{\eta }^{\prime }\left( i\right) },
\label{eigenvalue}
\end{equation}
where
\begin{equation} \label{legcolength}
l_{\eta }^{\prime }(i):=\#\left\{ j<i;\eta _{j}\geq \eta _{i}\right\}
+\#\left\{ j>i;\eta _{j}>\eta _{i}\right\} . 
\end{equation}

The second operator $\Phi _{q}$, a raising operator, transforms the
polynomial $E_{\eta }$ to $E_{\Phi \eta }.$ Here 
$$
\Phi \eta:=(\eta_2,\ldots,\eta_n,\eta_1+1)
$$
and the operator $\Phi _{q}$ is
defined by \cite{peter3}
\begin{equation*}
\Phi
_{q}:=z_{n}T_{n-1}^{-1}\ldots T_{1}^{-1}=t^{i-n}T_{n-1}\ldots T_{i}z_{i}T_{i-1}^{-1}\ldots T_{i}^{-1}.
\label{raising polynomials}
\end{equation*}
The operator $T_i^{-1}$ is related to $T_i$ by the quadratic relation in (\ref{hecke algebra}) and given explicitly by
$$
T_i^{-1}:=t^{-1}-1+t^{-1}T_i.
$$
The raising operator acts on $E_\eta$ according to \cite{peter3}
\begin{equation}
\Phi _{q}E_{\eta }(z)=t^{-\#\{i>1;\eta_i\leq \eta_1\} }E_{\Phi \eta } (z). \label{3.3marshall}
\end{equation}

By defining the nonsymmetric Macdonald polynomial $E_{(0,\ldots,0)}(z):=1$ and observing that every composition $\eta$ can be recursively generated from $(0,\ldots,0)$ using only $s_i$ and $\Phi$ we can be assured that each $E_\eta$ can be recursively generated using only $T_i$ and $\Phi_q$.

A further alternative characterization of the nonsymmetric Macdonald polynomials is as multivariate orthogonal polynomials. This definition requires an inner product and two partial orderings. Introduce the inner product 
\begin{equation}
\left\langle f,g\right\rangle _{q,t}:=\text{CT}[ f( z;q,t)
g( z^{-1},q^{-1},t^{-1}) W( z) ]
\label{inner product}
\end{equation}
where CT denotes the constant term in the formal Laurent expansion and 
\begin{equation*}
W( z) :=W( z;q,t) :=\prod_{1\leq
i<j\leq n}\frac{( \frac{z_{i}}{z_{j}};q) _{\infty }(q \frac{
z_{j}}{z_{i}};q) _{\infty }}{( t\frac{z_{i}}{z_{j}};q),
_{\infty }(qt\frac{z_{j}}{z_{i}};q) _{\infty }},
\end{equation*}
with the Pochhammer symbol defined by $\left( a;q\right) _{\infty }:=\prod_{j=0}^{\infty }\left( 1-aq^{j}\right)$.
Let $| \eta| :=\Sigma _{i=1}^{n}\eta _{i}$ denote the modulus of $\eta $. The dominance ordering $<$, a partial ordering on compositions of the same modulus, is defined by 
\begin{equation*} 
\eta < \mu \text{ iff }\eta \not = \mu \text{ and }\Sigma
_{i=1}^{p}\left( \eta _{i}-\mu _{i}\right) \geq 0 \text{ for all }1\leq p\leq n. 
\end{equation*}
A further partial ordering on compositions of the same modulus is $\prec $ is defined by
\begin{equation*}
\mu \prec \eta \text{ iff }\mu ^{+}<\eta ^{+}\text{ or in the case }\mu
^{+}=\eta ^{+},\text{ }\mu <\eta
\end{equation*}
where $\eta ^{+}$ is the unique partition obtained by permuting the
components of $\eta $. For a given value of $|\eta|$, the $E_\eta$ can be constructed via a Gram-Schmidt procedure with respect to (\ref{inner product}) from the requirements that \cite{affine}
\begin{equation}
\langle E_\eta, E_\nu \rangle_{q,t}=\delta_{\eta\nu}\,\mathcal{N}_\eta, \label{orthogonality relation}
\end{equation} 
and that with $z^\eta:=z_1^{\eta_1}z_2^{\eta_2}\ldots z_n^{\eta_n}$
\begin{equation}
E_{\eta }( z;q,t) =z^{\eta }+\sum_{\mu \prec \eta
}b_{\eta \mu }z^{\mu }, \hspace{1cm} b_{\eta \lambda }\in \mathbb{Q}( q,t).\label{forma}
\end{equation}

We will have future use for the explicit value of $\mathcal{N}_{\eta}$. To express this we require a number of quantities dependent on $\eta$. For each node $s=(i,j)\in \text{diag}(\eta)$ we define the arm length, $a_{\eta}(s):=\eta_{i}-j$, arm colength, $a'_{\eta}(s):=j-1$, leg length, $l_{\eta}(s):=\#\{k<i:j\leq\eta_{k}+1\leq\eta_{i}\}+\#\{k<i:j\leq\eta_{k}\leq\eta_{i}\}$ and leg colength $l'_{\eta}(s)$, given by (\ref{legcolength}). From these we define \cite{sahi}
\begin{align*}
d_{\eta }&:=d_{\eta }\left( q,t\right)=\prod_{s\in \text{diag}(\eta )}\left(
1-q^{a_{\eta }\left( s\right) +1}t^{l_{\eta }\left( s\right) +1}\right) 
\text{, \ \ } \\
d_{\eta }^{\prime }&:=d_{\eta }^{\prime }\left( q,t\right) =\prod_{s\in \text{diag}%
(\eta )}\left( 1-q^{a_{\eta }\left( s\right) +1}t^{l_{\eta }\left( s\right)
}\right) , \\
e_{\eta }&:=e_{\eta }\left( q,t\right) =\prod_{s\in \text{diag}(\eta )}\left(
1-q^{a_{\eta }^{\prime }\left( s\right) +1}t^{n-l_{\eta }^{\prime }\left(
s\right) }\right) \text{ }, \\
e_{\eta }^{\prime }&:=e_{\eta }^{\prime }\left( q,t\right) =\prod_{s\in \text{diag}(\eta
)}\left( 1-q^{a_{\eta }^{\prime }\left( s\right) +1}t^{n-1-l_{\eta }^{\prime
}\left( s\right) }\right).
\end{align*}
In this notation the explicit formula for $\mathcal{N}_\eta$ is given by \cite[see, e.g.,][]{cherednik}
\begin{equation}
\mathcal{N}_\eta=\frac{d'_{\eta } e_{\eta
}}{d_{\eta }e_{\eta }^{\prime }}{\left\langle
1,1\right\rangle _{q,t}}. \label{Nn}
\end{equation}

We now introduce a generalisation of the nonsymmetric Macdonald polynomials, the interpolation Macdonald polynomials.

\section{Interpolation Macdonald polynomials\label{insm}}
Although the interpolation Macdonald polynomials are not homogeneous, like the Macdonald polynomials they exhibit a certain triangular structure, specified by
 \begin{equation}
E_{\eta }^{\ast }( z) =z^{\eta }+\sum_{\lambda \prec \eta}b'_{\eta \lambda }z^{\lambda }, \hspace{1cm} b'_{\eta \lambda }\in \mathbb{Q}( q,t),\label{intform}
\end{equation}
where here $\lambda \prec \eta$ is extended to compositions $\lambda$ of modulus less than $|\eta|$ as well. Moreover, also in analogy with the nonsymmetric Macdonald polynomials, they permit a number of distinct characterisations which can be taken as their definition. Here we present such definitions, important to our subsequent workings.

Let
\begin{equation}
\Xi _{i}:=z_{i}^{-1}+z_{i}^{-1}H_{i}\ldots H_{n-1}\Phi H_{1}\ldots H_{i-1}.
\label{XiI}
\end{equation}
In (\ref{XiI}) 
\begin{align}
H_{i}:=&\frac{( t-1) z_{i}}{z_{i}-z_{i+1}}+\frac{z_{i}-tz_{i+1}}{
z_{i}-z_{i+1}}s_{i}  \label{hecke} \\
=&t+\frac{z_i-tz_{i+1}}{z_i-z_{i+1}}(s_{i}-1)
\end{align}
(cf. (\ref{Ti})) is a Hecke operator and 
\begin{equation}
\Phi :=\left( z_{n}-t^{-n+1}\right) \Delta \text{ },  \label{Phi}
\end{equation}%
where
\begin{equation*}
\Delta f\left( z_{1},...,z_{n}\right) =f\left( \frac{z_{n}}{q}%
,z_{1},...,z_{n-1}\right)
\end{equation*}%
is a raising operator.
We note that the Hecke operator, like the Demazure-Lustig operator is a realisation of the type-$A$ Hecke algebra.
\begin{definition} \rm{\cite{knop}}
With $\Xi _{i}$ as given in (\ref{XiI}) the interpolation Macdonald polynomials can be
defined, up to normalisation, as the unique simultaneous eigenfunctions of each $\Xi_i$ of the form (\ref{intform}) according to 
\begin{equation}
\Xi _{i}E_{\eta }^{\ast }\left( z;q,t\right) =\overline{\eta }%
_{i}^{-1}E_{\eta }^{\ast }\left( z;q,t\right) ,  \label{eigenfunction}
\end{equation}%
where $\overline{\eta }_{i}$ is given by (\ref{eigenvalue}). 
\end{definition}

In \cite{knop} Knop showed that the top homogeneous component of $E^*_\eta(z;q,t)$ is $E_\eta(z;q^{-1},t^{-1})$ using a relationship between the corresponding eigenoperators. This relationship is fundamental to our study, allowing us to use interpolation polynomial theory to derive explicit formulas for the coefficients $A_{\eta \lambda}^{(r)}$ in (\ref{nonzero}).

Another defining characteristic of the $E_\eta^*$ relates to the the recursive generation (cf (\ref{TiEn}),(\ref{3.3marshall})).

\begin{definition} \rm{\cite{knop}} With $H_i$ given by (\ref{hecke}) and $\Phi$ given by (\ref{Phi}) we can recursively generate $E_\eta^*$ recursively from $E_{(0,\ldots,0)}^*=1$ using
\begin{align}
E_{s_{i}\eta }^{\ast }( z)= &\left\{ 
\begin{tabular}{ll}
$H_{i}E_{\eta }^{\ast }( z)-\frac{t-1}{1-\delta _{i\eta }^{-1}( q,t) }E_{\eta }^{\ast
}( z)$ & $\eta _{i}<\eta
_{i+1}$ \\ 
$E_{\eta }^{\ast }( z) $ & $\eta _{i}=\eta _{i+1}$ \\ 
$\frac{( 1-\delta
_{i,\eta }( q,t)) ^{2} }{( 1-t\delta _{i,\eta
}( q,t)) (t- \delta _{i,\eta }( q,t))}\Big(H_{i}E_{\eta }^{\ast }( z)-\frac{t-1}{1-\delta _{i\eta }^{-1}( q,t) }E_{\eta }^{\ast
}( z)\Big)$ & $\eta _{i}>\eta _{i+1}$
\end{tabular}
\right.  \label{hecke on interpolation} \\ \notag \\
\Phi E_{\eta }^{\ast }( z)=&q^{-\eta _{1}}E_{\Phi \eta }^{\ast
}( z). \label{raising on interpolation}
\end{align}
\end{definition}

The final definition characterises the polynomials according to their vanishing properties.

\begin{definition} \rm{\cite{knop}}
With $\overline{\eta}_i$ as given in (\ref{eigenvalue}) we define $\overline{\eta}:=(\overline{\eta}_1,\ldots,\overline{\eta}_n)$. The interpolation Macdonald polynomials can be
defined, up to normalisation, as the unique polynomial of degree $\leq
\left\vert \eta \right\vert $ satisfying
\begin{equation*} \label{thevanishingconditions}
E_{\eta }^{\ast }( \overline{\mu }) =0,\text{ }\left\vert \mu
\right\vert \leq \left\vert \eta \right\vert ,\text{ }\mu \not=\eta
\end{equation*}
and $E_{\eta }^{\ast }( \overline{\eta }) \not=0.$
\end{definition}

It is well known \cite{knop} that the interpolation Macdonald polynomials $E_{\eta }^{\ast }( z) $
vanish on a larger domain than $\lambda \in \Lambda $ with $\left\vert
\lambda \right\vert \leq \left\vert \eta \right\vert ,$ $\lambda \neq \eta .$ These extra vanishing properties, that play an imperative role in the derivation of our coefficients, and will now be considered in some detail.

\subsection{Successors and the Extra Vanishing Theorem\label{the extra
vanishing theorem}}

To state the larger vanishing domain of the interpolation polynomials a further partial ordering is required. We
write $\eta \preceq' \lambda $, and say $\lambda$ is a successor of $\eta$, if there exists a permutation $\sigma $ such
that 
$$\eta _{i}<\lambda _{\sigma ( i) }\hspace{0.1cm} \text{ if }\hspace{0.1cm} i<\sigma (
i)\hspace{0.2cm}  \text{ and }\hspace{0.2cm} \eta _{i}\leq \lambda _{\sigma ( i) }\hspace{0.1cm} \text{ if } \hspace{0.1cm}i\geq \sigma (
i) .$$

We call $\sigma $ a defining permutation for $\eta \preceq' \lambda $ and write $\eta \preceq' \lambda;\sigma.$ It is important to note that defining permutations are not unique. For example $(1,2,1) \prec' (1,2,2)$ has defining permutations $(1,2,3)$ and $(3,2,1)$. However, there is only one defining permutation such that for all $i$ such that $\sigma(i)=i$ we have $\eta_i=\lambda_i$;  this defining permutation is to be denoted $\widehat{\sigma}$. 
 
In \cite{knop} Knop showed that if $\eta \not\preceq' \lambda $ then $E_{\eta }^{\ast
}( \overline{\lambda }) =0$ using the eigenoperator and the defining vanishing properties of the interpolation polynomials. In this section we prove this extra vanishing property using an alternative method that employs all three defining properties of the interpolation polynomials. The alternative method allows us to extend Knop's result as it can be used to show the converse is also true, that is if $E_{\eta }^{\ast
}( \overline{\lambda }) =0$ then $\eta \not\preceq' \lambda $. 
 
We begin as in \cite{knop} introducing the minimal elements lying above $\eta ,$ the $
\lambda $ such that $\eta \preceq' \lambda $ and $\left\vert \lambda
\right\vert =\left\vert \eta \right\vert +1$. Such compositions are denoted by $c_{I}( \eta ) $ and defined by
\begin{equation}
( c_{I}( \eta ) ) _{j}=\left\{ 
\begin{tabular}{ll}
$\eta _{t_{k+1}}$ & $;j=t_{k},$ if $k=1,\ldots,s-1$ \\ 
$\eta _{t_{1}}+1$ & $;j=t_{s}$ \\ 
$\eta _{j}$ & $;j\notin I,$
\end{tabular}
\ \right.   \label{cIn}
\end{equation}
where $I=\left\{ t_{1},\ldots,t_{s}\right\} \subseteq \left\{ 1,\ldots,n\right\} $
with $1\leq t_{1}<\ldots<t_{s}\leq n.$ More explicitly 
\begin{eqnarray*}
c_{I}( \eta ) &=&(\eta _{1}\ldots\eta _{t_{1-1}},\mathbf{\eta }
_{t_{2}},\eta _{t_{1}+1}\ldots\eta _{t_{2}-1},\mathbf{\eta }_{t_{3}},\eta
_{t_{2}+1}\ldots \\
&&\hspace{1cm}\eta _{t_{s-1}-1},\mathbf{\eta }_{t_{s}},\eta _{t_{s-1}+1}\ldots\eta
_{t_{s}-1},\mathbf{\eta }_{t_{1}}\mathbf{+1},\eta _{t_{s}+1}\ldots\eta _{n}),
\end{eqnarray*}
(the $1$ added to $\eta_{t_1}$ has been set in bold to highlight its location).
We now show how each successor can be recursively generated from $\eta$ using the switching and raising operators.
\begin{proposition}
\label{successors by permutations}With $c_{I}( \eta ) $ defined
as above we have
\begin{equation*}
c_{I}( \eta ) =\sigma _{t_{1}+1}\ldots\sigma _{n}\Phi
s_{1}\ldots s_{t_{1}-1}\eta ,
\end{equation*}
where 
\begin{equation*}
\sigma _{i}=\left\{ 
\begin{tabular}{ll}
$1$ & $;i\in I$ \\ 
$s_{i-1}$ & $;i\not\in I.$
\end{tabular}
\right.
\end{equation*}
\end{proposition}

\begin{proof}
The operators to the right of $\Phi $ move $\eta _{t_{1}}$ to the first
position, thus enabling $\Phi$ to increase its value by $1$. Each $
s_{j-1}$ on the left hand side moves each $\eta _{j},$ for $j\not\in I,$ back
to its original position, automatically placing the $\eta _{i}$ with $i\in
I$ into the correct position.
\hfill $\square$ \end{proof}\begin{example}
Take $\eta =( 1,3,5,7,9,11,13,15) $ and $I=\left\{
3,4,5,7\right\} $
\begin{eqnarray*}
s_{2}\eta &=&( 1,\mathbf{5,3},7,9,11,13,15) \\
s_{1}s_{2}\eta &=&( \mathbf{5,1},3,7,9,11,13,15) \\
\Phi s_{1}s_{2}\eta &=&( 1,3,7,9,11,13,15,\mathbf{6}) \\
s_{7}\Phi s_{1}s_{2}\eta &=&( 1,3,7,9,11,13,\mathbf{6,15}) \\
s_{5}s_{7}\Phi s_{1}s_{2}\eta &=&( 1,3,7,9,\mathbf{13,11},6,15). \\
&&
\end{eqnarray*}
\end{example}
By the definition of $c_I(\eta)$ it is clear that $\eta \prec' c_{I}( \eta ) .$ The following lemma considers the other direction. 

\begin{lemma} \label{successorsprec} {\rm \cite{knop}}
\label{successor} If $|\lambda| = |\eta|+1$ and $\eta \prec' \lambda$ then there exists a set $I=\{t_1,\ldots,t_s\}\subseteq\{1,\ldots,n\}$ with $1\leq t_1<\ldots <t_s \leq n$ such that $c_I(\eta)=\lambda.$
\end{lemma}

\begin{proof} Since $|\lambda| = |\eta|+1$ and $\eta \prec' \lambda$ the defining permutation $\widehat{\sigma}$ must satisfy $\lambda_{\widehat{\sigma}{(i)}}=\eta_i$ for all but one $i$, say $i=k$, in which case $\lambda_{\widehat{\sigma}{(k)}}=\eta_k+1$. By the definition of $\prec'$ we must have $i\geq \widehat{\sigma}(i)$ for $i\not=k$. It follows that with with $I=\{i;\widehat{\sigma}(i)\not=i\}=\{t_1,\ldots,t_s\}$ we must have $\widehat{\sigma}$ specified by 
\begin{equation}
\widehat{\sigma}(i)=\left\{ 
\begin{tabular}{ll}
$i$ & $;i \not\in I$ \\ 
$t_{j-1}$ & $;i=t_j\in I, \hspace{1cm}j=2,\ldots,s,$ \\ 
$t_s$ & $;i=t_1 \in I.$
\end{tabular}
\ \right.  \label{definingpermutation}
\end{equation}
Combining (\ref{definingpermutation}) with (\ref{cIn}) shows $\lambda=c_I(\eta)$.
\hfill $\square$ \end{proof}

This lemma, in addition to the knowledge of the action of permutations on compositions 
\begin{equation}  \label{tupleoncomposition}
\sigma\eta=(\eta_{\sigma^{-1}(1)},\ldots,\eta_{\sigma^{-1}(n)}),
\end{equation} 
can be used to show that one does not need to check all permutations to establish that $\eta \preceq' \lambda $. Before stating the result - due to Knop \cite{knop} but to be derived differently below - we define $\omega _{\eta }$ to be the shortest
permutation such that 
\begin{equation} \label{shortestperm}
\omega _{\eta }^{-1}( \eta ) =\eta ^{+}.
\end{equation}

\begin{lemma} \label{definingpermutation2}For $\lambda$ such that $|\lambda| = |\eta|+1$ the defining permutation $\widehat{\sigma}$ of $\eta \preceq' \lambda $ is $\widehat{\sigma}=\omega_{\lambda} \omega_\eta^{-1}$.
\end{lemma}
\begin{proof} From the previous lemma we can replace $\lambda$ by $c_I(\eta)$ and specify the defining permutation $\widehat{\sigma}$ by (\ref{definingpermutation}). From the definition of $c_I(\eta)$ it is clear that 
\begin{equation*}
(\eta^+)_i\leq(c_I(\eta)^+)_i, \hspace{1cm} \text{for all }i\in\{1,\ldots,n\}.
\end{equation*}
Manipulating this using (\ref{shortestperm}) and (\ref{tupleoncomposition}) shows
\begin{equation*}
\eta_i\leq c_I(\eta)_{\omega_{c_I(\eta)} \omega_\eta^{-1}(i)}, \hspace{1cm} \text{for all }i\in\{1,\ldots,n\}.
\end{equation*}
By (\ref{shortestperm}) it can be deduced that $t_{j-1}=\omega_{c_I(\eta)} \omega_\eta^{-1}(t_j)$ for $j=\{2,\ldots, s\}$, $t_s=\omega_{c_I(\eta)} \omega_\eta^{-1}(t_1)$ and $\omega_{c_I(\eta)} \omega_\eta^{-1}(i)=i$ if $\eta_i=(c_I(\eta))_i$. These properties of $\omega_{c_I(\eta)} \omega_\eta^{-1}$ show it to be identically equal to (\ref{definingpermutation}), thus concluding the proof.
\hfill $\square$ \end{proof}

From Lemma \ref{definingpermutation2} we have the following corollaries. 
\begin{corollary} \label{compostoinofprec} If $\eta \preceq'  c_{I_1}(\eta);\sigma$ and $c_{I_1}(\eta) \preceq' c_{I_2}c_{I_1}(\eta);\rho$ then $ \Rightarrow \eta\preceq' c_{I_2}c_{I_1}(\eta);\rho \circ \sigma$.
\end{corollary}

\begin{corollary} \rm{\cite{knop}} For $\lambda$ such that $|\lambda| = |\eta|+r$ the defining permutation of $\eta \preceq' \lambda $ where $\sigma(i)=i$ if $\eta_i=\lambda_i$ is $\sigma=\omega_{\lambda} \omega_\eta^{-1}$.

\end{corollary}
\begin{corollary}  \label{genp}
If $|\lambda| = |\eta|+r$ and $\eta \prec' \lambda$ then there exists sets  $\{I_1,\ldots,I_r \}$, $I_k=\{{^kt}_{1},\ldots,{^kt}_{s}\}\subseteq\{1,\ldots,n\}$ with $1\leq {^kt}_{1}<\ldots <{^kt}_{s} \leq n$ such that $c_{I_r}\ldots c_{I_1}(\eta)=\lambda.$
\end{corollary}

We now work towards showing the main theorem of the section, Proposition \ref{extra vanishing condition} below, by first considering $\lambda$ such that $|\lambda|=|\eta|+1$.

\begin{proposition} \label{plus1}
For $\lambda$ such that  $|\lambda|=|\eta|+1$ we have $E_{\eta }^{\ast }( \overline{\lambda }) =0$ if and
only if $\eta\not\preceq' \lambda.$
\end{proposition}

\begin{proof}
Rewriting the eigenoperator $\Xi_i$ as
\begin{equation*}
z_{i}\Xi _{i}-1=H_{i}\ldots H_{n-1}\Phi H_{1}\ldots H_{i-1}
\end{equation*}
and making note of the recursive generation formulas $( \ref{raising on
interpolation}) $ and $( \ref{hecke on interpolation}) $
shows us that
\begin{align}
(z_{i}\Xi _{i}-1)E_{\eta }^{\ast }( z) =&H_{i}\ldots H_{n-1}\Phi
H_{1}\ldots H_{i-1}E_{\eta }^{\ast }( z) \notag \\
( \overline{\lambda }_{i}\overline{\eta }_{i}^{-1}-1) E_{\eta
}^{\ast }( \overline{\lambda }) =&\sum_{\nu:\nu=c_I(\eta)
}c_{\eta
\nu }E_{\nu }^{\ast }( \overline{\lambda }),  \hspace{1cm}c_{\eta
\nu } \in \mathbb{Q}(q,t), \label{eopact}
\end{align}
where the summation restriction to $\nu=c_I(\eta)$ in (\ref{eopact}) is a consequence of Proposition \ref{successors by permutations}. The vanishing conditions of $
E_{\nu }^{\ast }$ imply that $E_{\eta }^{\ast }( \overline{\lambda 
}) =0$ if and only if $\lambda \not=c_{I}( \eta )$, that is, if and only if $\eta\not\preceq' \lambda.$
\hfill $\square$ \end{proof}

\begin{proposition}\label{extra vanishing condition}
We have $E_{\eta }^{\ast }( \overline{\lambda }) =0$ if and
only if $\eta\not\preceq' \lambda.$
\end{proposition}
\begin{proof}
Here we prove the equivalent statement $E_{\eta }^{\ast }( \overline{\lambda }) \not=0$ if and
only if $\eta\preceq' \lambda$. We begin with $\lambda$ such that $|\lambda|=|\eta|+2$. By Corollary \ref{genp} we know that $\lambda=c_{I_2} c_{I_1}(\eta)$ for some sets $\{I_1,I_2 \}$, $I_k=\{{^kt}_{1},\ldots,{^kt}_{s}\}\subseteq\{1,\ldots,n\}$ with $1\leq {^kt}_{1}<\ldots <{^kt}_{s} \leq n$. 

Taking $i={^1t}_{1}$ in (\ref{eopact}) gives
\begin{equation*}
( \overline{\lambda }_{{^1t}_{1}}\overline{\eta }_{{^1t}_{1}}^{-1}-1) E_{\eta
}^{\ast }( \overline{\lambda }) =\sum_{\nu:\nu=c_I(\eta)}\widetilde{c}_{\eta
\nu }E_{\nu }^{\ast }( \overline{\lambda }),  \hspace{1cm}\widetilde{c}_{\eta
\nu } \in \mathbb{Q}(q,t).
\end{equation*}
Proposition \ref{successors by permutations} can be used to show $E_{c_{I_1}(\eta)}$ is in the summation and by Proposition \ref{plus1} we know $E_{c_{I_1}(\eta)}(\overline{\lambda})\not=0.$ We can be sure that $( \overline{\lambda }_{{^1t}_{1}}\overline{\eta }_{{^1t}_{1}}^{-1}-1)\not=0$ since even if $\lambda_{{^1t}_{1}}=\eta_{{^1t}_{1}}$ we would still have either $\overline{\lambda}=\overline{\eta}/t$ or $\overline{\lambda}=\overline{\eta}/t^2$ due to the increased value of $l'_\eta({^1t}_{1}).$ These results together show $E_{\eta }^{\ast }( \overline{\lambda }) \not=0$ if $\eta\preceq' \lambda.$ For the converse we again use (\ref{eopact}), this time however taking $i$ to be the position of the leftmost component of $\eta$ that does not occur with the same frequency in $\lambda,$ 
\begin{equation}
( \overline{\lambda }_{i}\overline{\eta }_{i}^{-1}-1) E_{\eta
}^{\ast }( \overline{\lambda }) =\sum_{\nu:\nu=c_I(\eta)}\widehat{c}_{\eta
\nu }E_{\nu }^{\ast }( \overline{\lambda }),  \hspace{1cm}\widehat{c}_{\eta
\nu } \in \mathbb{Q}(q,t). \label{eopact2} 
\end{equation}
By assumption $E_{\eta
}^{\ast }( \overline{\lambda })\not=0$ and as before $( \overline{\lambda }_{i}\overline{\eta }_i^{-1}-1)\not=0$, therefore there exists a $E_\nu(\overline{\lambda})$ on the RHS of (\ref{eopact2}) that does not vanish and by Proposition \ref{plus1} we have $\nu\preceq'\lambda$. Since $\nu=c_I(\eta)$ implies $\eta\preceq'\nu$ we can use Corollary \ref{compostoinofprec} to show $\eta\preceq'\lambda$, which completes the proof for the case $|\lambda|=|\eta|+2$. Applying this procedure iteratively shows the result holds for general $\lambda$.
 \hfill $\square$ \end{proof}

%One of the key features of the derivation we use below is that $\lambda$ is obtained from $\eta$ via $c_I(\eta)$. It was the observation that each $\lambda$ such that $E^*_\eta(\overline{\lambda})\not=0$ could be generated by repeated action of $c_I$'s as specified by Corollary \ref{genp} that lead to the consideration of the Pieri-type coefficients for $r>1$. Experimentation with the connections between the results in this section and general Pieri-type coefficients gave more motivation than results. The alternative derivations of the extra vanishing condition have been included to provide readers with an alternative viewpoint that may lead to fruitful research. \footnote{not sure if this kind of thing is appropriate in papers, but I like acknowledging that this was the starting point for the study}
We now move onto the major goal of the paper, deriving explicit formulas for the Pieri-type coefficients.

\section{Structure of the Pieri-type expansions for the nonsymmetric Macdonald polynomials}\label{structure}
Before focussing on the Pieri-type formulas for the nonsymmetric Macdonald polynomials we consider a more general situation.

Take a linear basis of polynomials $\{P_\eta(z);\, \eta\in \mathbb{N}^n\}$ of degree $|\eta|$, satisfying 
$$P_\eta(\overline{\mu})=0, \hspace{1cm} \text{ for all }|\mu|\leq|\eta|, \,\,\,\, \mu\not=\eta.
$$
The interpolation Macdonald polynomials are clearly such a basis, further examples are the Schubert polynomials \cite{lascouxschubert} and the interpolation Jack polynomials for appropriate meaning of $\overline{\mu}$. If we take a polynomial $f_r(z)$ of degree $r$ such that $f_r(\overline{\eta})=0$ we have
\begin{equation*}
f_r(z)P_\eta(z)=\sum_{|\eta|+1\leq| \lambda |\leq |\eta|+r}a_{\eta \lambda}P_ \lambda(z). \label{genvan}
\end{equation*}

We seek the most simplified expressions for the coefficients in the expansion in the case that $f_r(z)=(e_r(z)-e_r(\overline{\mu}))$ and $P_\eta(z)=E^*_\eta(z;q,t)$. Following the strategies of \cite{marshalljack} we then use the fact that the top homogeneous component of $E_\eta^*(z;q,t)$ is $E_\eta(z;q^{-1},t^{-1})$ to conclude that the coefficients $A_{\eta,\lambda}^{(r)}(q,t)$ in 

\begin{equation}
(e_r(z)-e_r(\overline{\eta}))E_{\eta }^{\ast }( z;q,t) =\sum_{|\eta|+1\leq| \lambda |\leq |\eta|+r} A_{\eta \lambda
}^{(r) }( q,t)  E_{\lambda }^{\ast }(
z;q,t).\label{basicint}
\end{equation}
are the same as those in (\ref{nonzero}).
%\begin{equation}\label{nsbegin}
%e_r(z)E_{\eta }( z;q^{-1},t^{-1}) =\sum_{|\lambda |= |\eta|+r} A_{\eta \lambda
%}^{(r) }( q,t)  E_{\lambda }(
%z;q^{-1},t^{-1}).
%\end{equation}
The theories of both the nonsymmetric Macdonald polynomials and the interpolation polynomials are employed to simplify the coefficients $A_{\eta \lambda}^{(r)}$. 

We begin using the interpolation polynomials to restrict the summation in both (\ref{basicint}) and (\ref{nonzero}). The extra vanishing conditions of the interpolation polynomials state that $E_\eta^*(\overline{\lambda})=0$ if $\eta \not\preceq' \lambda $ allowing us to write
\begin{equation}
(e_r(z)-e_r(\overline{\eta}))E_{\eta }^{\ast }( z;q,t) =\sum_{\substack{|\eta|+1\leq| \lambda |\leq |\eta|+r\\ \eta\preceq'\lambda}} A_{\eta \lambda
}^{(r) }( q,t)  E_{\lambda }^{\ast }(
z;q,t)\label{basicint2}
\end{equation}
and
\begin{equation*}\label{nsbegin2}
e_r(z)E_{\eta }( z;q^{-1},t^{-1}) =\sum_{\substack{|\lambda |= |\eta|+r \\ \eta \preceq' \lambda} } A_{\eta \lambda
}^{(r) }( q,t)  E_{\lambda }(
z;q^{-1},t^{-1})
\end{equation*}
respectively. Following the methods of Forrester and McAnally \cite{forresterandmcanally} we exploit the orthogonality of the nonsymmetric Macdonald polynomials to identify further zero coefficients in the expansions.

\begin{proposition}
\label{relating coefficients}We have
\begin{equation}
A_{\eta \lambda }^{(r) }( q,t) =A_{\lambda ,\eta
+( 1^{n}) }^{(n-r )}( q^{-1},t^{-1}) 
\frac{\mathcal{N}_{\eta }}{\mathcal{N}_{\lambda }},
\label{coefficient relationship}
\end{equation}
where $\mathcal{N}_{\eta }$ is given by $( \ref{Nn})$ and $\eta +(k^{n}):=(\eta_1+k,\ldots, \eta_n+k)$.
\end{proposition}

\begin{proof}
With $\lambda $ such that $| \lambda | =| \eta
| +r$ and $\eta \preceq' \lambda $ by $( \ref{nonzero}) $
and the orthogonality and linearity properties of $\left\langle \cdot ,\cdot
\right\rangle _{q,t}$ $( \ref{orthogonality relation}) $ we have 
\begin{align}
&\left\langle e_r(z)E_{\eta }( z;q^{-1},t^{-1})
,E_{\lambda }( z;q^{-1},t^{-1}) \right\rangle _{q,t}   \notag \\
&\hspace{2cm}=A_{\eta\lambda }^{(r)}( q,t)
\left\langle E_{\lambda }( z;q^{-1},t^{-1}),E_{\lambda }(
z;q^{-1},t^{-1}) \right\rangle _{q,t}.  \label{first orth lemma}
\end{align}
Using $( \ref{inner product}) $ we can write the left hand side
of $( \ref{first orth lemma}) $ as 
\begin{equation*}
\text{CT}[  e_r(z)E_{\eta }( z;q^{-1},t^{-1})
E_{\lambda }( z^{-1};q,t) W( z) ] .
\end{equation*}
Replacing $z$ with $z^{-1}$ and then multiplying $e_r(z^{-1})$
by $( z_{1}\ldots z_{n}) $ and $E_{\eta }(
z^{-1};q^{-1},t^{-1}) $ by $( z_{1}\ldots z_{n}) ^{-1}$ gives 
\begin{equation*}
\text{CT}[ e_{n-r}(z) E_{\eta +( 1^{n}) }(
z^{-1};q^{-1},t^{-1}) E_{\lambda }( z;q,t) W(
z^{-1}) ] ,
\end{equation*}
This can be recognised as
\begin{equation*}
\left\langle  e_{n-r}(z) E_{\lambda }( z;q,t) ,E_{\eta
+( 1^{n}) }( z;q,t) \right\rangle _{q,t}.
\end{equation*}
Again, by linearity and orthogonality we have 
\begin{equation*}
A_{\lambda ,\eta +( 1^{n}) }^{(n-r ) }(
q^{-1},t^{-1}) \left\langle E_{\eta +( 1^{n}) }( z;q,t)
,E_{\eta +( 1^{n}) }( z;q,t) \right\rangle _{q,t}.
\end{equation*}
Since $\mathcal{N}_{\eta }( q,t) =\mathcal{N}_{\eta }(
q^{-1},t^{-1}) $ and $\mathcal{N}_{\eta }( q,t) =\mathcal{N}
_{\eta +( 1^{n}) }( q,t) $ it follows that $( \ref
{coefficient relationship}) $ is true.
\hfill $\square$ \end{proof}

\begin{corollary}\label{furthvan}
We have
\begin{equation*}\label{vanishcoeff}
A_{\eta \lambda }^{(r)}( q,t)=0 \hspace{0.25cm} \text{ if }\hspace{0.25cm}
\eta \not\preceq' \lambda \hspace{0.25cm} \text{or} \hspace{0.25cm} \lambda \not\preceq' \eta +( 1^{n}),
\end{equation*}
and therefore\begin{equation}
e_r(z)E_{\eta }( z;q^{-1},t^{-1}) =\sum_{\substack{|\lambda |= |\eta|+r \\ \eta \preceq' \lambda \preceq'\eta+(1^n)} }A_{\eta \lambda }^{(r)}(
q,t) E_{\lambda }( z;q^{-1},t^{-1}). \label{nonzerocoefficientspieri}
\end{equation}
\end{corollary}

In \cite{forresterandmcanally} Forrester and McAnally gave further structure to the $\lambda$ in (\ref{nonzerocoefficientspieri}), showing that compositions $\lambda$ satisfying  $\eta \preceq' \lambda \preceq'\eta+(1^n)$ are characterised by the properties that there are sets $\{i_1,\ldots,i_r  \}\subset\{1,\ldots,n\}$ and $\{j_1,\ldots,j_{n-r}  \}=\{1,\ldots,n\}\setminus\{1,\ldots,n\} $ such that
\begin{align*} \label{specialcomposition}
\lambda_{\sigma(i_p)}&=\eta_{i_p}+1\hspace{1cm}\text{for }i_p\leq \sigma(i_p)\hspace{1cm}p=1,\ldots,r \notag \\
\lambda_{\sigma(j_p)}&=\eta_{j_p}\hspace{1cm}\text{for }j_p\leq \sigma(j_p)\hspace{1cm}j=1,\ldots,n-r
\end{align*} 
for some defining permutation $\sigma$.

%Since the top homogeneous component of $(e_r(z)-e_r(\overline{\eta}))E^*_\eta(z)$ is $e_r(z)E_\eta(z)$ we must also have 
%\begin{equation}
%e_r(z)E^*_{\eta }( z;q,t) =\sum_{i=1}^r \sum_{\lambda \in 
%\mathbb{J}_{\eta ,i} }A_{\eta \lambda }^{(r)}(
%q,t) E^*_{\lambda}( z;q,t) \label{nonzerocoefficientspieriint}.
%\end{equation}

\section{An alternative derivation of the Pieri-type formulas for $r=1$}\label{pieri1}

The methods used to determine the general Pieri-type coefficients are motivated by those of Lascoux \cite{lascoux}, and in particular are quite different from those used in \cite{wendyb}. Lascoux computed the Pieri-type coefficients for the case $r=1$ by observing that 
\begin{equation*}
(z_1+\ldots + z_n-|\overline{\eta}|)E_\eta^*(z)=\sum_\lambda c_\eta^\lambda E_\lambda^*(z)
\end{equation*} 
could be evaluated at $\overline{\lambda}$ to give
\begin{equation*}
c_\eta^\lambda=\frac{(|\overline{\lambda}|-|\overline{\eta}|)E_\eta^*(\overline{\lambda})}{E_\lambda^*(\overline{\lambda})}.
\end{equation*}
Then an explicit formula for 
\begin{equation} \label{ratioofevaluatedmacs}
\frac{E_\eta^*(\overline{\lambda})}{E_\lambda^*(\overline{\lambda})}
\end{equation}
was found using an inductive proof stemming from the recursive generation of $E_\lambda^*$ from $E_\eta^*$. 
We first present an alternative derivation of the evaluation of (\ref{ratioofevaluatedmacs}), where in contrast to Lascoux's inductive proof we use the recursive generation of the $\lambda$ from $\eta$ specified by Proposition \ref{successors by permutations} and also the eigenoperator properties of the $E^*_\eta$. This strategy will further be used to give a derivation of the general Pieri-type coefficients.

We begin by giving a more explicit description of compositions in the summation of (\ref{nonzerocoefficientspieri}) for the case $r=1$. Let $I=\{ t_{1},\ldots,t_{s}\} $ with $1\leq
t_{1}<\ldots<t_{s}\leq n$ and $I\neq \emptyset .$ The set $I$ is said to be maximal with
respect to $\eta $ iff
\begin{equation*}
\begin{tabular}{lll}
$( 1) $ & $\eta _{j}\neq \eta _{t_{u}},$ & $\text{ }
j=t_{u-1}+1,\ldots,t_{u}-1$ $( u=1,\ldots,s;\text{ }t_{0}:=0) ;$ \\ 
$( 2) $ & $\eta _{j}\neq \eta _{t_{1}}+1,$ & $\text{ }
j=t_{s}+1,\ldots,n.$
\end{tabular}
\end{equation*}
For example with $\eta=(1,1,1)$ the sets $\{1\}$, $\{1,2\}$ and $\{ 1,2,3\}$ are maximal and all other subsets of $\{ 1,2,3\}$ are not. Knop \cite{knop} showed that the compositions in the summation of (\ref{nonzerocoefficientspieri}) with $r=1$ can be specified as
$$
\{ \lambda: \lambda=c_I(\eta),\hspace{0.1cm} I \hspace{0.1cm}{\rm maximal} \}.
$$

With this in place we are now ready to give the first step in determining an explicit formula for (\ref{ratioofevaluatedmacs}).

\begin{proposition}\label{ratio}
With $I=\{ t_{1},\ldots,t_{s}\} $ maximal with respect to $\eta $
and $\lambda =c_{I}( \eta ) $
\begin{equation}
H_{t_{1}}\ldots H_{n}\Phi H_{1}\ldots H_{t_{1}-1}E_{\eta }^{\ast }( \overline{
\lambda }) =\frac{q^{-\eta _{t_{1}}}\delta ( \eta ,I) \beta
( \eta ,I) (\overline{\lambda }_{t_{1}}\overline{\eta }
_{t_{1}}^{-1}-1)}{(t-1)}E_{\eta }^{\ast }( \overline{\lambda }) ,
\label{required formula}
\end{equation}
where 
\begin{equation*}
\delta ( \eta ,I) :=\prod_{u=1}^{s}\frac{t-1}{1-\overline{
\lambda }_{t_{u}}\overline{\eta }_{t_{u}}^{-1}}  \label{delta new}
\end{equation*}
and 
\begin{equation}
\beta ( \eta ,I) :=\prod_{i}\frac{( X( i)
-t) ( tX( i) -1) }{( X( i)
-1) ^{2}},  \label{beta new}
\end{equation}
with $X( i) :=$ $\overline{\eta }_{t_{u}( i) }
\overline{\eta }_{i}^{-1},$ where $t_{u}( i) $ is the first
element in $I$ above $i$ and if $i>t_{s}$ then $\eta _{t_{u}( i)
}=\eta _{t_{1}}+1.$ The product $( \ref{beta new}) $ is over all $
i\not\in I$ with $\eta _{i}>\eta _{t_{u}( i) }.$
\end{proposition}

\begin{proof}
With $I$ maximal with respect to $\eta$, Proposition \ref{successors by permutations} can be used to show the polynomial $E^*_\lambda$ occurs exactly once in the expansion of \begin{equation}
H_{t_{1}}\ldots H_{n}\Phi H_{1}\ldots H_{t_{1}-1}E_{\eta }^{\ast }( z).
\label{expansion}
\end{equation}
Since all polynomials $E_{\nu }^{\ast }( z) $
appearing in the full expansion of $( \ref{expansion}) $ are of
size $| \eta | +1$, by the vanishing conditions of $
E^{\ast }( z) $, evaluating $( \ref{expansion}) $ at $
\overline{\lambda }$ will reduce it to some
multiple of $E_{\lambda }^{\ast }( \overline{\lambda }) .$ 
 
We begin by expanding $( \ref{expansion}) $ from the right using
the recursive generation formulae to determine the coefficient of each $
E_{\mu ( j) }^{\ast }( z) ,$ where $\mu (
j) $ represents the transformed $\eta $ after the $j^{th}$ step in
the transformation from $\eta $ to $\lambda $. Since the operators $H_{i}$
and $\Phi $ commute through constants we consider the coefficient
contribution of each operator on the $E_{\mu ( i) }^{\ast }(
z) $ to observe the result more easily. First consider the expansion
of 
\begin{equation*}
H_{1}\ldots H_{t_{1}-1}E_{\eta }^{\ast }( z) .
\end{equation*}
We know from Proposition \ref{successors by permutations} that we require
the $s_{i}$ to act on $\eta $ at every stage to move $\eta _{t_{1}}$ to the
first position before acting upon by $\Phi .$ Therefore we must take the
coefficient of $E_{s_{i}\mu ( t_{1}-i) }^{\ast }( z) $
when $H_{i}$ acts on each $E_{\mu ( t_{1}-i) }^{\ast }(
z) .$ At each stage the switching operator swaps $\eta _{j}$ with $
\eta _{t_{1}}$, where $j$ runs from $t_{1}-1$ to $1.$ By $( \ref{hecke
on interpolation}) $ the coefficient contribution will be $1$ if $\eta
_{j}<\eta _{t_{1}}$, and 
\begin{equation*}
\frac{( X( j) -t) ( tX( j) -1) }{
( X( j) -1) ^{2}}
\end{equation*}
where $X( j) =\overline{\eta }_{t_{1}}\overline{\eta }_{j}^{-1}$
if $\eta _{j}>\eta _{t_{1}}.$ Multiplying these terms together gives the
coefficient of $E_{\mu ( t_{1}-1) }^{\ast }( z) $
where 
\begin{equation*}
\mu ( t_{1}-1) =( \eta _{t_{1}},\eta _{1},\ldots,\eta
_{t_{1}-1},\eta _{t_{1}+1},\ldots.,\eta _{n}) .
\end{equation*}

Next we act upon $H_{1}\ldots H_{t_{1}-1}E_{\eta }^{\ast }( z) $ with 
$\Phi .$ By $( \ref{raising on interpolation}) $ when $\Phi $
acts on $E_{\mu ( t_{1}-1) }^{\ast }( z) $ the
coefficient contribution is $q^{-\eta _{t_{1}}}$ and the new polynomial is $
E_{\mu ( t_{1}) }^{\ast }( z) $ where 
\begin{equation*}
\mu ( t_{1}) =( \eta _{1},\ldots,\eta _{t_{1}-1},\eta
_{t_{1}+1},\ldots.,\eta _{n},\eta _{t_{1}}+1) .
\end{equation*}
We proceed by considering the coefficients of $E_{\mu ( j)
}^{\ast }$ for $j>t_{1}$ in the expansion of 
\begin{equation*}
H_{t_{1}}\ldots H_{n-1}E_{\mu
( t_{1}) }^{\ast }( z) .
\end{equation*}
At this stage particular attention must be payed to the set $I$ to know whether we want to extract the
coefficient and polynomial of $E_{\mu ( t_{1}+j) }^{\ast }(
z) $ or $E_{s_{n-j}\mu ( t_{1}+j) }^{\ast }( z) $
from the action of $H_{n-j}$ on $E_{\mu ( t_{1}+j) }^{\ast
}( z) .$ First consider the action of $H_{k-1}$ for $k\in I,$ $
k>t_{1}.$ From Proposition \ref{successors by permutations} we know that if $
k\in I$ we don't require the switch $s_{k-1}$ in the generation of $\lambda $. Therefore, when $H_{k-1}$ acts on $E_{\mu ( t_{1}+n-k+1) }^{\ast }$,
we take the coefficient of $E_{\mu ( t_{1}+n-k+1) }^{\ast }.$ By $
( \ref{hecke on interpolation}) $ this is given by 
\begin{equation*}
\frac{t-1}{1-\delta _{\mu ( t_{1}+n-k+1) ,k-1}^{-1}}.
\end{equation*}
To determine the value of $\delta _{\mu ( t_{1}+n-k+1) ,k}^{-1}$
we consider the $( k-1) ^{th}$ and $k^{th}$ value of $\mu (
t_{1}+n-k+1) .$ Since we do not need to swap the components we must
have $\mu ( t_{1}+n-k+1) _{k}=\lambda _{k}.$ Also, at this stage $
\mu ( t_{1}+n-k+1) _{k-1}$ is equal to $\eta _{k}$ since the $
( k-1) ^{th}$ component hasn't changed since $\mu (
t_{1}) .$ Hence the coefficient of the polynomial $E_{\mu (
t_{1}+n-k+1) }^{\ast }$ in $H_{k-1}E_{\mu ( t_{1}+n-k+1)
}^{\ast }$ with $k\in I$ is 
\begin{equation}
\frac{t-1}{1-\overline{\lambda }_{k}\overline{\eta }_{k}^{-1}}.
\label{new delta part}
\end{equation}
It is important to note here that $\overline{\lambda }_{k}\not=\overline{
\eta }_{k}$ for $k\in I$ since even if $\lambda _{k}=\eta _{k}$ we have $
\overline{\lambda }_{k}=\overline{\eta }_{k}/t$. The total contribution of
these terms is the product of $( \ref{new delta part}) $ as $k$
runs from $2$ to $s.$ 

Lastly we consider the case where $k\not\in I,$
where we take the coefficient of the $E_{s_{k-1}\mu (
t_{1}+n-k+1) }^{\ast }$ in the expansion of $H_{k-1}E_{\mu (
t_{1}+n-k+1) }^{\ast }.$ For $k>t_{s}$ we use $
s_{k-1}$ to move $\eta _{t_{1}}+1$ to the $t_{s}^{th}$ position,
each time swapping $\eta _{t_{1}}+1$ with $\eta _{k}.$ For $t_{1}<k<t_{s}$
each $s_{k-1}$ is used to move $\eta _{t_{u}}$ to the $
t_{u-1}^{th}$ position, swapping $\eta _{k}$ and $\eta
_{t_{u}( k) },$ where $t_{u}( k) $ is the first
element of $I$ above $k.$ By (\ref{hecke on interpolation}) when either $
\eta _{k}>\eta _{t_{1}}+1$, for $j=t_{s}+1,\ldots,n$ or $\eta _{k}>\eta
_{t_{u}( k) },$ for $t_{1}<k<t_{s}$ we have the coefficient 
\begin{equation*}
\frac{( X( k) -t) ( tX( k) -1) }{
( X( k) -1) ^{2}}
\end{equation*}
where $X( i) =\overline{\eta }_{t_{u}( k) }\overline{
\eta }_{k}^{-1}$. Combining all coefficients gives
$( \ref{required formula}) $. 
\hfill $\square$ \end{proof}

\begin{proposition} \label{interpolationratioprop}
With $I=\{ t_{1},\ldots,t_{s}\} $ maximal with respect to $\eta$ 
and $\lambda
=c_{I}( \eta ) $  
\begin{equation}
\frac{E_{\eta }^{\ast }( \overline{\lambda }) }{E_{\lambda
}^{\ast }( \overline{\lambda }) }=\frac{q^{-\eta _{t_{1}}}\delta
( \eta ,I) \beta ( \eta ,I) }{( 1-t) }.
\label{final ratio}
\end{equation}
\end{proposition}

\begin{proof}
By $( \ref{XiI}) $ and $( \ref{eigenfunction}) $ we
have
\begin{equation}
( z_{t_{1}}\overline{\eta }_{t_{1}}^{-1}-1) E_{\eta }^{\ast
}( z) =H_{t_{1}}\ldots H_{n-1}\Phi H_{1}\ldots H_{t_{1}-1}E_{\eta }^{\ast
}( z) .  \label{nearly done}
\end{equation}
By the previous proposition and the vanishing properties of the $E^*$, evaluating  $( \ref{nearly done}
) $ at $\overline{\lambda }$ gives (\ref{final ratio}).
\hfill $\square$ \end{proof}

\begin{corollary}\label{pierireq1}
We have 
\begin{equation*}
e_1(z) E_{\eta }( z;q^{-1},t^{-1})
=\sum_{\lambda :\lambda =c_{I}( \eta ) }\frac{(
| \overline{\lambda }| -| \overline{\eta }
| ) q^{-\eta _{t_{1}}}\delta ( \eta ,I) \beta
( \eta ,I) }{( 1-t) }E_{\lambda }(
z;q^{-1},t^{-1}) .  \label{alternate r=1}
\end{equation*}
\end{corollary}

It is straightforward to check that this formula and that obtained in Proposition 8 of \cite{wendyb} are equivalent.  We now show how this procedure can be extended to determine the general Pieri-type coefficients.

\section{The General Pieri-Type Formula Coefficients}\label{pierir}

To determine explicit formulas for the $A_{\eta \lambda}^{(r)}(q,t)$ in (\ref{nonzerocoefficientspieri}) we once again return to the theory of the interpolation polynomials. We begin by rewriting (\ref{basicint2}) as 

\begin{equation}
e_r(z)E^*_{\eta }( z;q,t) =\sum_{i=1}^r\sum_{\eta \preceq' \lambda^i }A_{\eta \lambda^i }^{(r)}(
q,t) E^*_{\lambda^i }( z;q,t) \label{nonzerocoefficientspieriint2},
\end{equation}
where we've introduced the notation $\lambda^i$ to denote a composition of modulus $|\eta|+i$.
 
Since the sum in (\ref{nonzerocoefficientspieriint2}) is over compositions of varying modulus we cannot just evaluate at each $\overline{\lambda}$ to obtain the coefficient of $E_\lambda^*(z)$ like we did in the proof of Proposition \ref{ratio}. Here, the coefficients must be generated recursively beginning with $\lambda$ such that $|\lambda|=|\eta|+1$. The details are provided in the following result.

\begin{theorem} \label{generalpiericoefficients}
For $\eta \preceq' \lambda^i$ the
coefficients $A_{\eta \lambda^i}^{(r)}$ in (\ref{nonzerocoefficientspieriint2}) 
%is nonzero and 
are recursively generated as
\begin{align}
A_{\eta \lambda^1}^{(r)}&=\frac{(e_r(\overline{\lambda^1})-e_r(\overline{\eta}))E_\eta^*(\overline{\lambda^1})}{E_ {\lambda^1} ^*(\overline{\lambda^1})} \label{firstcoefficientgeneralpieri}\\
A_{\eta \lambda^2}^{(r)}&=\frac{(e_r(\overline{\lambda^2})-e_r(\overline{\eta}))E_\eta^*(\overline{\lambda^2})}{E_{\lambda^2}^*(\overline{\lambda^2})}-\sum_{\lambda^1: \eta \preceq' \lambda^1 \preceq' \lambda^2}A_{\eta \lambda^1}^{(r)}\frac{E^*_{\lambda^1}(\overline{\lambda^2})}{E^*_{\lambda^2}(\overline{\lambda^2})}  \label{2ndcoefficientgeneralpieri}
\end{align}
and in general
\begin{equation} \label{generalcoefficientgeneralpieri}
A_{\eta \lambda^i}^{(r)}=\frac{(e_r(\overline{\lambda^i})-e_r(\overline{\eta}))E_\eta^*(\overline{\lambda^i})}{E_{\lambda^i}^*(\overline{\lambda^i})}-\sum_{k=1}^{i-1}\sum_{\lambda^k: \eta \preceq' \lambda^k \preceq' \lambda^i}A_{\eta \lambda^k}^{(r)}\frac{E^*_{\lambda^k}(\overline{\lambda^i})}{E_{\lambda^i}^*(\overline{\lambda^i})}.
\end{equation}
If $\eta \not\preceq' \lambda^i$ we have $A_{\eta \lambda^i}^{(r)}=0.$
\end{theorem}
\begin{proof}
We first consider the structure of the coefficients. By the vanishing properties of $E^*_{\lambda^i}$ evaluating (\ref{nonzerocoefficientspieriint2}) at $\overline{\lambda^1}$ gives (\ref{firstcoefficientgeneralpieri}). When we evaluate (\ref{nonzerocoefficientspieriint2}) at $\overline{\lambda^2}$ we obtain
\begin{equation} \label{2ndcoefficientgeneralpieriderivation}
(e_r(\overline{\lambda^2})-e_r(\overline{\eta}))E_\eta^*(\overline{\lambda^2})=A_{\eta \lambda^2}^{(r)}E^*_{\lambda^2}(\overline{\lambda^2})+\sum_{\lambda^1:\eta \preceq' \lambda^1 \preceq \lambda^2 }A_{\eta \lambda^1}^{(r)}E^*_{\lambda^1}(\overline{\lambda^2}), 
\end{equation}
since any $E_{\lambda^1}$ such that $\lambda^1 \preceq \lambda^2 $ will not vanish when evaluated at $\overline{\lambda^2}$. Rearranging (\ref{2ndcoefficientgeneralpieriderivation}) gives (\ref{2ndcoefficientgeneralpieri}). The general coefficient formula (\ref{generalcoefficientgeneralpieri}) is derived using the same methods, recursively generating $A_{\eta \lambda^1}^{(r)}, \dots, A_{\eta \lambda^{i-1}}^{(r)}$ to determine  $A_{\eta \lambda^{i}}^{(r)}$.

The claim that for $\eta \not\preceq' \lambda^i$ we have $A_{\eta \lambda^i}^{(r)}=0$ is a due to the vanishing properties of $E^*_{\lambda^i}$ and Corollary \ref{compostoinofprec}. If $\eta \not\preceq' \lambda^i$ the vanishing properties would cause the leading term of (\ref{generalcoefficientgeneralpieri}) to vanish and by Corollary \ref{compostoinofprec} there would be no such $\lambda^k$ that satisfied $ \eta \preceq' \lambda^k \preceq' \lambda^i$.
\hfill $\square$ \end{proof} 

\begin{corollary}
The coefficients $A_{\eta \lambda^r}^{(r)}$ in (\ref{nonzerocoefficientspieri}), where $\eta\preceq'\lambda^r\preceq'\eta+(1^n)$, satisfy the recursion (\ref{generalcoefficientgeneralpieri}) with $i=r,$
\begin{equation*}\label{nsbegin3}
A_{\eta \lambda^r}^{(r)}=\frac{(e_r(\overline{\lambda^r})-e_r(\overline{\eta}))E_\eta^*(\overline{\lambda^r})}{E_{\lambda^r}^*(\overline{\lambda^r})}-\sum_{k=1}^{r-1}\sum_{\lambda^k: \eta \preceq' \lambda^k \preceq' \lambda^r}A_{\eta \lambda^k}^{(r)}\frac{E^*_{\lambda^k}(\overline{\lambda^r})}{E_{\lambda^r}^*(\overline{\lambda^r})}.
\end{equation*}
\end{corollary}

To use this to obtain explicit formulas for the coefficients $A_{\eta \lambda^i}^{(r)}$ we require formulas for the evaluation of $E^*_{\lambda^i}(\overline{\lambda^j})$, where $i<j$. The evaluation of $E_{\lambda^i}^*(\overline{\lambda^i})$ follows from the general formula \cite{wendyb}  
$$
E^*_\eta(\overline{\eta})=d_\eta'(q^{-1},t^{-1})\prod_{i=1}^n\overline{\eta}_i^{\eta_i}.
$$

\begin{proposition} \label{binomialcoeff}
Let $D_{I_k}(\eta):=\sigma_{k}\ldots \sigma_{n-1}\Phi \sigma_{1}\ldots \sigma_{k-1}(\eta)$ where $I_k\subseteq\{1,\ldots,n-1\}$ and $\sigma_j=s_j$ if $j\in I_k$. With $D_k(\lambda^i):=\{\nu: \nu=D_{I_k}(\eta)\}$ for some $I_k \subseteq\{1,\ldots,n-1\}$, and $\lambda^i \preceq' \lambda^j$ with $i<j-1$, we have
\begin{equation*} \label{computingallbinomials}
E_{\lambda^i}^*(\overline{\lambda^j})=\sum_{\substack{\nu\in D_k(\lambda^i),  \\ \nu \preceq' \lambda^j}}\frac{\overline{\nu}_k\overline{\lambda^i}_k^{-1}-1}{\overline{\lambda^j}_k\overline{\lambda^i}_k^{-1}-1}\frac{E_{\lambda^i}^*(\overline{\nu})}{E_\nu^*(\overline{\nu})}E^*_\nu(\overline{\lambda^j}),
\end{equation*}
where k is the position of the leftmost component of $\lambda^i$ that does not occur with the same frequency in $\lambda^j.$
\end{proposition}

\begin{proof} The case where $j=i+1$ is given in Proposition \ref{interpolationratioprop}, and so we begin with $j=i+2$. Manipulating (\ref{XiI}) and acting on $E^*_{\lambda^i}$ gives
\begin{align}
(z_k\Xi_k-1)E^*_{\lambda^i}(z)=&\,H_k\ldots H_{n-1}\Phi H_1 \ldots H_{k-1}E^*_{\lambda^i}(z) \notag \\
=&\sum_{\nu\in D_k(\lambda^i)}c_{\eta,\nu}^kE^*_\nu(z), \label{44'}
\end{align}
where $k$ is specified above. Evaluating at $\overline{\nu}$ for $\nu$ a particular composition in the sum shows   
\begin{equation*}
c_{\eta,\nu}^k=(\overline{\nu}_k\overline{\lambda^i}_k^{-1}-1)\frac{E_{\lambda^i}^*(\overline{\nu})}{E_{\nu}^*(\overline{\nu})}.
\end{equation*}
Substituting back in (\ref{44'})  shows
\begin{equation} \label{togetbinomial}
E^*_{\lambda^i}(z)=\sum_{\nu\in D_k(\lambda^i)}\frac{(\overline{\nu}_k\overline{\lambda^i}_k^{-1}-1)}{(z_k\overline{\lambda^i}_k^{-1}-1)}\frac{E_{\lambda^i}^*(\overline{\nu})}{E_{\nu}^*(\overline{\nu})}E^*_\nu(z).
\end{equation}

We compute $E_{\lambda^i}^*(\overline{\lambda^{i+2}})$ as follows. With $k$ as specified we can be sure there is at least one $\nu \in D_k({\lambda^i})$ such that $E_{\nu}^*(\overline{\lambda^{i+2}})\not = 0$. Evaluating (\ref{togetbinomial}) at $\overline{\lambda^{i+2}}$ gives 
\begin{equation} \label{togetbinomial2}
E_{\lambda^i}^*(\overline{\lambda^{i+2}})=\sum_{\substack{\nu\in D_k(\lambda^i)\\\nu \preceq' \lambda^{i+2}}}\frac{\overline{\nu}_k\overline{\lambda^i}_k^{-1}-1}{\overline{\lambda^{i+2}}_k\overline{\lambda^i}_k^{-1}-1}\frac{E_{\lambda^i}^*(\overline{\nu})}{E_\nu^*(\overline{\nu})}E^*_\nu(\overline{\lambda^{i+2}}),
\end{equation}
where the further restriction on the summation to $\nu \preceq \lambda^{i+2}$ is due to the vanishing conditions of $E^*_\nu$. Since Proposition \ref{interpolationratioprop} gives an explicit formula for each $E^*_\nu(\overline{\lambda^{i+2}})$ (\ref{togetbinomial2}) does indeed give us an explicit formula for $E_{\lambda^i}^*(\overline{\lambda^{i+2}})$. 
 
One can then evaluate (\ref{togetbinomial}) at $\overline{\lambda^{i+3}}$ and use $E_{\lambda^i}^*(\overline{\lambda^{i+2}})$ to find an explicit formula for $E_{\lambda^i}^*(\overline{\lambda^{i+3}})$. This process can be extended to allow any $E_{\lambda^i}^*(\overline{\lambda^{j}})$ where $i<j-1$ to be broken down into a combination of evaluations of the form $E_{\lambda^l}^*(\overline{\lambda^{l+1}})$, which in turn can be explicitly evaluated using (\ref{final ratio}).
%us a sum of evaluations $ Hence by evaluating (\ref{togetbinomial}) at $\overline{\lambda^{i+2}}$ with  and using  to compute each $E_{\nu}^*(\overline{\lambda^{i+2}})$ gives an explicit formula for $E_{\lambda^i}^*(\overline{\lambda^{i+2}})$.  can be computed using (\ref{computingallbinomials}). Applying this procedure recursively one can break down any $E_{\lambda^i}^*(\overline{\lambda^{j}})$ where $i<j-1$ to a combination of evaluations of the form $E_{\lambda^l}^*(\overline{\lambda^{l+1}})$.

\hfill $\square$ \end{proof}

This result leads us very nicely to a consequence for the generalised binomial coefficients $\binom{\nu }{\eta }_{q,t}$. These coefficients are due to Sahi \cite{sahi} and are given in terms of the interpolation Macdonald polynomials by
\begin{equation*}
\binom{\eta }{\nu }_{q,t}:=\frac{E_{\eta }^{\ast }( \overline{
\nu}) }{E_{\nu }^{\ast }( \overline{\nu }) }.
\label{generalised q-binomial}
\end{equation*}
In \cite{wendyb} an explicit formula for the case where $|\nu|=|\eta|+1$ was given, however in general there is no known explicit formula.

\begin{corollary} \label{binomcoeffexp}With $i<j-1$ we have 
\begin{equation*} \label{computingallbinomialscoeff}
\binom{\lambda^i }{\lambda^j }_{q,t}=\frac{E_{\lambda^i}^*(\overline{\lambda^j})}{E_{\lambda^j}^*(\overline{\lambda^j})}=\sum_{\substack{\nu\in D_k(\lambda^i) \\ \nu \preceq' \lambda^j}}\frac{\overline{\nu}_k\overline{\lambda^i}_k^{-1}-1}{\overline{\lambda^j}_k\overline{\lambda^i}_k^{-1}-1}\frac{E_{\lambda^i}^*(\overline{\nu})}{E_\nu^*(\overline{\nu})}\frac{E^*_\nu(\overline{\lambda^j})}{E_{\lambda^j}^*(\overline{\lambda^j})},
\end{equation*}
 where k is the position of the leftmost component of $\lambda^i$ that does not occur with the same frequency in $\lambda^j.$
\end{corollary}

Clearly the explicit formulas for the Pieri-type coefficients are rather complex. We complete the study by considering possible simplifications.

\section{Simplifying the Pieri coefficients}

First a coefficient of unity in each Pieri-type formula is identified. An analogous result was observed  earlier by Forrester and McAnnaly \cite{forresterandmcanally} within Jack polynomial theory, and identical principles apply for Macdonald polynomials. 

Forrester and McAnnaly found that with $\eta+\chi_r$ given by
\begin{equation*}
(\eta+\chi_r)_i:=\left\{ 
\begin{tabular}{ll}
$\eta_i,$ & $l_\eta'(i)\geq r$ \\ 
$\eta_i+1,$ & $l_\eta'(i)< r$,
\end{tabular}
\right. 
\end{equation*}
we have
\begin{equation*}
A_{\eta,\eta+\chi_r}^{(r)}=1.
\end{equation*}
We first give an explicit derivation in the case $r=1$ and then state their reasoning in the general case. 
\begin{proposition} 
We have
\begin{equation}\label{46'}
A_{\eta,\eta+\chi_1}^{(1)}=1.
\end{equation}
\end{proposition}
\begin{proof}
Let $\eta_i$ be such that $l_\eta'(i)=0$. By (\ref{hecke on interpolation}) and (\ref{raising on interpolation}) the coefficient of $E^*_{\eta+\chi_1}(z)$ in the expansion of $H_i\ldots H_{n-1} \Phi H_1\ldots H_{i-1}E_\eta^*(z)$ will be $q^{-\eta_i}$. Using the vanishing properties of $E^*$ and 
\begin{equation*} \label{difviewofeigenoperator}
z_{i}\Xi _{i}-1=H_{i}\ldots H_{n-1}\Phi H_{1}\ldots H_{i-1}
\end{equation*}
 we have
\begin{equation*}
\frac{E^*_{\eta}(\overline{\eta+\chi_1})}{E^*_{\eta+\chi_1}(\overline{\eta+\chi_1})}=\frac{1}{(q-1)q^{\eta_i}}=\frac{1}{e_1(\overline{\lambda^1})-e_1(\overline{\eta})},
\end{equation*}
which upon substitution in (\ref{firstcoefficientgeneralpieri}) implies (\ref{46'})
\hfill $\square$ \end{proof}
\begin{proposition} \rm{\cite{forresterandmcanally}} We have
\begin{equation*}
A_{\eta,\eta+\chi_r}^{(r)}=1.
\end{equation*}
\end{proposition}
\begin{proof}
By definition of the Macondald polynomials the coefficient of $z^\eta$ in $E_\eta(z)$ is unity, and consequently the coefficient of $z^{\eta+\chi_r}$ in $e_r(z)E_\eta(z)$ is unity also. Since $\lambda^r \prec \eta+\chi_r$ for all $\lambda^r\not=\eta+\chi_r$ such that $\eta \preceq' \lambda^r \preceq' \eta+(1^n)$ the triangular structure of Macdonald polynomials (\ref{forma}) ensures that the monomial $z^{\eta+\chi_r}$ will only occur in $E_{\eta+\chi_r(z)}$, forcing $A_{\eta,\eta+\chi_r}^{(r)}$ to be unity.
\hfill $\square$ \end{proof}

We can also greatly simplify the coefficients $A_{\eta\lambda^{r}}^{r}(q,t)$ in the case where $r>\lceil \frac{n}{2}\rceil$. By (\ref{coefficient relationship}) we can reduce $A_{\eta\lambda^{r}}^{r}(q,t)$ to 
$$
A_{\lambda^r,\eta+(1^n)}^{(n-r)}(q^{-1},t^{-1})\frac{\mathcal{N}_\eta}{\mathcal{N}_{\lambda^r}}.
$$  
This is a simplification as there are less steps required to obtain $A^{(r-n)}$ than $A^{(r)}$ when $r>\lceil \frac{n}{2}\rceil$.
 
%We note that the results of Section \ref{structure} that if one was to require a Pieri-type coefficient for $r>\lceil \frac{n}{2}\rceil$ it would be more efficient to first find $A_{\eta\lambda^{n-r}}^{n-r}$ and then use (\ref{coefficient relationship}) to obtain $A_{\eta\lambda^{r}}^{r}$.

There is some freedom in the implementation of the recurrences, and we have investigated ways to reduce the required number of calculations. For example, the $k$ specified in the formulas of Proposition \ref{binomialcoeff} is not the only such $k$ that will provide a pathway to the explicit formula of $E^*_{\lambda^i}(\overline{\lambda^j})$, where $i<j$. The only requirement on $k$ is that $D_k(\lambda^i)$ contains a $\nu$ such that $\nu \preceq\lambda^{j}$. If one was to compute $E^*_{\lambda^i}(\overline{\lambda^j})$ it would be most efficient to choose $k$ such that the number of $\nu\in D_k(\lambda^i)$ such that $\nu\not\preceq' \lambda^j$ is maximised and thus minimising the number of computations by increasing the number of vanishing terms. At this stage there doesn't seem to be an obvious way of choosing such a $k$ and the problem remains open.

It is clear from trial computations that the Pieri-type coefficients can sometimes be expressed as a product. This is always true for the case $r=1$. Unfortunately our recursive formulas shed no light on the general requirement for a product formula to hold true. 

With this last point in mind we conclude our discussion by giving one further expression for the Pieri-type coefficients in the case $r=1$. These formulas are analogous to those given in \cite{forresterandmcanally} for the Jack polynomials. We are aware of some Ansatz analysis into generalisations of these formulas for cases beyond $r=1$, and include these formulas to hopefully motivate further investigations. We note also that the formulas given in \cite{forresterandmcanally} contain some typographical errors. They can be corrected following the derivation provided here.

Earlier we stated that the formulas for $A^{(1)}_{\eta \lambda}$ obtained in Proposition \ref{interpolationratioprop} were equivalent to those found in \cite{wendyb}. Here we state the latter formulas and then show how they can be expressed in a form suitable for generalisation. Define
$$
\widehat{a}(x,y):=\frac{(t-1)x}{x-y}, \hspace{0.5cm}\widehat{b}(x,y):=\frac{x-ty}{x-y},
$$
and with $I=\{t_1,\ldots,t_s\}$ such that $1 \leq t_1<\ldots<t_s\leq n,$ define
\begin{align*}
A_I(z):=&\widehat{a}\left( \frac{z_{t_{s}}}{q}%
,z_{t_{1}}\right) \prod\limits_{u=1}^{s-1}\widehat{a}\left(
z_{t_{u}},z_{t_{u+1}}\right), \\
\widetilde{B}_I\left( z\right)
:=&\prod\limits_{u=1}^{s}\prod\limits_{j=t_{u-1}+1}^{t_{u}-1}\widehat{b}%
\left( z_{t_{u}},z_{j}\right) \prod\limits_{j=t_{s}+1}^{n}\widehat{b}\left(
qz_{t_{1}},z_{j}\right)  \\
&\times \left( qz_{t_{1}}-t^{-n+1}\right) ,\text{ }t_{0}:=0.
\end{align*}
In this notation it was shown in \cite{wendyb} that
\begin{equation}\label{newpieri}
e_1(z)E_\eta(z;,q^{-1},t^{-1})=\sum_{\lambda=c_I(\eta)}\frac{(1-q)d_\eta'(q^{-1},t^{-1})A_I(\overline{\eta})\widetilde{B}_I(\overline{\eta})}{d_\lambda'(q^{-1},t^{-1})q^{t_1+1}(t-1)}E_\lambda(z;,q^{-1},t^{-1}).
\end{equation}
%We introduce 
%$$
%\widehat{B}(z):=-t^{-n+1}\prod\limits_{u=1}^{s}\prod\limits_{j=t_{u-1}+1}^{t_{u}-1}\widehat{b}%
%\left( z_{t_{u}},z_{j}\right) \prod\limits_{j=t_{s}+1}^{n}\widehat{b}\left(
%qz_{t_{1}},z_{j}\right),\text{ }t_{0}:=0. 
%$$
%Using the identities \cite{sahi}
%$$
%\frac{e'_{\Phi \eta}(q,t)}{e_\eta'(q,t)}=1-q^{\eta_1+1}t^{n-1-l'_\eta(1)} \hspace{0.25cm} \text{and}\hspace{0.25cm}e'_{s_i\eta}(q,t)=e_\eta(q,t),
%$$
%we see that 
%$$
%\widehat{B}_I(\overline{\eta})=\frac{e'_\eta}{e'_{c_I(\eta)}}\widetilde{B}_I(\overline{\eta}), 
%$$
%and hence
%$$
%e_1(z)E_\eta(z;,q^{-1},t^{-1})=\frac{d_\eta'(q^{-1},t^{-1})}{e_\eta'(q,t)}\sum_{\lambda=c_I(\eta)}\frac{e_{\lambda}'(q,t)(1-q)A_I(\overline{\eta})\widehat{B}_I(\overline{\eta})}{d_\lambda'(q^{-1},t^{-1})q^{t_1+1}(t-1)}E_\lambda(z;,q^{-1},t^{-1}).
%$$

By introducing the sets $G_0$ and $G_1$
\begin{align*}
G_0&:=G_0(\eta,\lambda)=\{i\in\{1,\ldots, n\}; \lambda_{\sigma(i)}=\eta_i \},
\\ G_1&:=G_1(\eta,\lambda)=\{i\in\{1,\ldots, n\};\lambda_{\sigma(i)}=\eta_i+1 \},
\end{align*}
where $\sigma$ is the defining permutation of $\eta\preceq'\lambda$, we can rewrite $A_I(\overline{\eta})$ and $\widetilde{B}_I(\overline{\eta})$ as
\begin{align*}
A_I(\overline{\eta})B_I(\overline{\eta})=&\prod_{\sigma(j)<j}\frac{(t-1)\overline{\eta}_{\sigma(j)}}{\overline{\eta}_{\sigma(j)}-\overline{\eta}_{j}}\bigg( \prod_{j\in G_1}\frac{(t-1)\overline{\eta}_{\sigma(j)}}{\overline{\eta}_{\sigma(j)}-q \overline{\eta}_{j}} (q\overline{\eta}_j-t^{-n+1}) \bigg)\\
&\times \prod_{\sigma(j)<k<j}\frac{\overline{\eta}_j-t\overline{\eta}_k}{\overline{\eta}_j-\overline{\eta}_k}\prod_{j\in G_1, k<j}\frac{\overline{\eta}_j-t\overline{\eta}_k}{\overline{\eta}_j-t\overline{\eta}_k}\prod_{\substack{k\in G_0,j\in G_1\\\sigma(j)<k}}\frac{q\overline{\eta}_j-t\overline{\eta}_k}{q\overline{\eta}_j-\overline{\eta}_k}.
\end{align*}

These formulas can be substituted in (\ref{newpieri}) to give a new viewpoint on the Pieri-type coefficients for $r=1$.

\subsection*{Acknowledement.}
I am very thankful for my supervisor Peter Forrester's support. Throughout the past two years he has taught the importance of persistence in mathematic research, and this work is a consequence of such teachings. I thank Alain Lascoux for useful discussion. This work was supported by an APA scholarship, the ARC and the University of Melbourne.

\bibliographystyle{plain}

\end{document}